\documentclass[12pt]{amsart}
\usepackage[T1]{fontenc}
\usepackage[colorlinks,linkcolor=blue,citecolor=blue,urlcolor=red]{hyperref}

\newcommand{\tto}{\dashrightarrow}
\newcommand{\by}{\xrightarrow}
\newcommand{\inj}{\hookrightarrow}

\newcommand{\surj}{\rightarrow\!\!\!\!\!\rightarrow}
\newcommand{\Surj}{\relbar\joinrel\surj}

\newcommand{\car}{\operatorname{char}}
\newcommand{\can}{{\operatorname{can}}}
\newcommand{\corank}{\operatorname{corank}}

\newcommand{\et}{{\operatorname{\acute{e}t}}}
\newcommand{\Q}{\mathbb{Q}}

\newcommand{\Z}{\mathbb{Z}}

\newtheorem{thm}{Theorem}

\newtheorem{lemma}{Lemma}
\newtheorem{conj}{Conjecture}
\theoremstyle{remark}
\newtheorem{rk}{Remark}

\begin{document}
\title{On the universal regular homomorphism in codimension $2$}
\author{Bruno Kahn}
\address{IMJ-PRG\\Case 247\\
4 place Jussieu\\
75252 Paris Cedex 05\\France}
\email{bruno.kahn@imj-prg.fr}
\date{May 18, 2020}
\begin{abstract}
We point out a gap in Murre's proof of the existence of a universal regular homomorphism for codimension $2$ cycles on a smooth projective variety, and offer two arguments to fill this gap.
\end{abstract}
\maketitle

%\enlargethispage*{20pt}

In \cite{sitges}, Jacob Murre shows the existence of a universal regular homomorphism for algebraically trivial cycles of codimension $2$ on a smooth projective variety over an algebraically closed field. This theorem has been largely used in the literature, most lately in \cite{achteretal}, \cite{colliot-pirutka} and \cite{benoist-witt}; for example, it is essential in \cite{benoist-witt} for descending the method of Clemens and Griffiths \cite{c-g} to non-algebraically closed fields, thus allowing Benoist and Wittenberg to obtain new examples of geometrically rational nonrational $3$-folds. 

Unfortunately its proof contains a gap, but fortunately this gap can be filled, actually by two different methods. This is the purpose of this note, which is a slight modification of a letter to Murre on December 5, 2018.\\

%Dear Jacob,

%Here are the details I promised on this small gap I found in \cite{sitges}.

Recall the set-up, with the notation of \cite{sitges}: $V$ is a smooth projective variety over an algebraically closed field $k$ and $A^n(V)$ denotes the group of codimension $n$ cycles algebraically equivalent to $0$ on $V$, modulo rational equivalence. Following Samuel, given an abelian $k$-variety $A$, a homomorphism
\[\phi:A^n(V)\to A(k) \]
is said to be \emph{regular} if, for any pointed smooth projective $k$-variety $(T,t_0)$ and any correspondence $Z\in CH^n(T\times V)$, the composition
\begin{equation} \label{eq1}
T(k)\by{w_Z}A^n(V)\by{\phi} A(k)
\end{equation}
is induced by a morphism $f:T\to A$; here $w_Z$ is the composition
\begin{equation}\label{eq0}
T(k)\to A_0(T)\by{Z_*} A^n(V)
\end{equation}
where the first map sends $t$ to $[t]-[t_0]$. (Note that $f$ is then unique, by Zariski density of the rational points in $T$.)

Using fancy language, regular homomorphisms from $A^n(V)$ form a category and a \emph{universal regular homomorphism} is an initial object of this category, if it exists. This initial object is well-known to exist when $n=0$, $n=1$ (the Picard variety) and $n=\dim X$ (the Albanese variety). Murre's theorem is:

\begin{thm}[\protect{\cite[Th. 1.9]{sitges}}] A universal regular homomorphism $\phi_0$  exists when $n=2$ for any $V$ (of dimension $\ge 2$).
\end{thm}

Recall the main steps of his proof. First, given a regular homomorphism $\phi$, its image in $A(k)$ is given by the points of some sub-abelian variety $A'\subseteq A$ \cite[Lemma 1.6.2 i)]{sitges}. From this, one deduces \cite[Prop. 2.1]{sitges} that $\phi_0$ exists if and only if $\dim A$ is bounded when $\phi$ runs through the \emph{surjective} regular homomorphisms. Now, Murre's key idea is to bound  $\dim A$ by the torsion of $A^2(V)$, which is controlled by the Merkurjev-Suslin theorem (Bloch's observation).

Let us elaborate a little on this point, to avoid the $l$-adic argument of \emph{loc. cit.}: it suffices to prove that $\phi$ induces a surjection
\begin{equation}\label{eq2}
A^2(V)\{l\} \Surj A(k)\{l\}
\end{equation}
for some prime $l\ne \car k$, where $M\{l\}$ denotes the $l$-primary torsion of an abelian group $M$: indeed, $\corank A(k)\{l\} = 2\dim A$. Mainly by Merkurjev-Suslin (Diagram in \cite[Prop. 6.1]{sitges})\footnote{One could replace this diagram by the injection $CH^2(V)\inj H^4_\et(V,\Gamma(2))$ of \cite[Th. 2.13 (c)]{licht}, together with the surjection $H^3_\et(V,\Q_l/\Z_l(2))\surj H^4_\et(V,\Gamma(2))\{l\}$, cf. loc. cit., proof of Th. 2.15; here, $\Gamma(2)$ is Lichtenbaum's complex.},  
\[\corank CH^2(V)\{l\} \le \corank H^3_\et(V,\Q_l/\Z_l(2)) (= b_3(V))\]
so the same holds \emph{a fortiori} for $\corank A^2(V)\{l\}$. 

Now, in \cite[Lemma 1.6.2 ii)]{sitges}, Murre constructs an abelian variety $B$ (pointed at $0$) and a correspondence $Z\in CH^2(B\times V)$ such that \eqref{eq1} is surjective for $T=B$. Since this map is induced by a morphism of abelian varieties sending $0$ to $0$ (hence a homomorphism), it restricts to a surjection 
\begin{equation}\label{eq3}
B\{l\}\surj A\{l\}.
\end{equation}

This allows me to explain
\begin{center}
the gap:
\end{center}

A priori \eqref{eq3} does not imply \eqref{eq2}, because $w_Z$ is in general only a set-theoretic map, not a group homomorphism (see e.g. \cite[Th. (3.1) a)]{bloch}).

We now fix a surjective regular homomorphism $\phi$ as above. We shall give two ways to fill this gap:

\enlargethispage*{20pt}

\begin{enumerate}
\item[(A)] construct $(B,Z)$ such that $w_Z$ is a homomorphism;
\item[(B)] prove that $w_Z$ always sends torsion to torsion.
\end{enumerate}

(A) was my initial idea, and (B) was inspired by a discussion with Murre.

\subsection*{Explanation of (A)} We have
%Take $(T,t_0,Z)$ with $T$ of dimension $1$. From \eqref{eq1} and \eqref{eq0}, we obtain a homomorphism
%\begin{equation}\label{eq4} 
%\phi\circ Z_*:A_0(T)=J(k)\to A(k)
%\end{equation}
%where $J=J(T)$ is the Jacobian of $T$.

%\begin{lemma}\label{p1} The homomorphism \eqref{eq4} is of the form $\phi\circ w_\alpha$ for some correspondence $\alpha\in CH^2(J\times V)$.
%\end{lemma}

\begin{lemma}\label{p1} Take $(T,t_0,z)$ with $T$ of dimension $1$ and $z\in CH^2(T\times V)$. Let $J=J(T)$ be the jacobian of $T$. Then the homomorphism $z_*:A_0(T)=J(k)\to A^2(V)$ is of the form $w_\alpha$ for some correspondence $\alpha\in CH^2(J\times V)$ (using $0\in J(k)$ as base point).
\end{lemma}

\begin{proof} Let $g$ be the genus of $T$. Recall from \cite[Ex. 3.12]{milne} the universal relative Cartier divisor $D_\can$ on $T\times T^{(g)}/T^{(g)}$, parametrising the effective divisors of degree $g$ on $T$. It defines a correspondence $D_\can:T^{(g)}\to T$. %, such that the composition of correspondences
%\[T\by{\theta} T^{(g)}\by{D_\can} T\]
%equals $\Delta_T +(g-1)[t_0]$, where $\theta(t) = (t,t_0,\dots,t_0)$. Here $[t_0]$ denotes the class of the composition $T\to \Spec k\by{t_0} T$. 
Composing with the graph of the birational map $J\tto T^{(g)}$ inverse to $(t_1,\dots,t_g)\mapsto \sum t_i -gt_0$, we find a (Chow) correspondence $D:J\to T$. %such that the composition
%\[T\by{\theta'} J\by{D'} T\]
%equals $\Delta_T +(g-1)[t_0]$, where $\theta'(t)=t-t_0$. 
I claim that $\alpha= z\circ D$ answers the question. Indeed, one checks immediately that the homomorphism
\[D_*:A_0(J)\to A_0(T)\]
is the Albanese morphism for $J$; hence the composition
\[J(k)\to A_0(J)\by{D_*} A_0(T)\]
is the identity.
\end{proof}

\begin{rk}\label{r1} On the other hand, the morphism $T\to A$ given by the regularity of $\phi$ factors through a homomorphism
\begin{equation}\label{eq5} J(T)\to A.
\end{equation}
 This homomorphism coincides with the one underlying $\phi\circ z_*$ in view of Lemma \ref{p1}. Indeed, by uniqueness, it suffices to see that \eqref{eq5} induces $\phi\circ z_*$ on $k$-points; this is clear since $T(k)$ generates $J(T)(k)$ as an abelian group.
\end{rk}

Consider all triples $(T,t_0,z)$ with $\dim T=1$. The homomorphism $\bigoplus A_0(T)\allowbreak\by{(z_*)} A^2(V)$ is surjective, hence so is $\bigoplus A_0(T)\by{(z_*)} A^2(V){\surj} A(k)$. As in Remark \ref{r1}, each summand of this homomorphism is induced by a homomorphism $\rho_{T,t_0,z}:J(T)\to A$, so 
\[B:=\prod_{(T,t_0,z)\in S} J(T)\by{(\rho_{T,t_0,z})} A\]
is surjective (faithfully flat) for a suitable finite set $S$. For each $(T,t_0,z)$, let $\alpha=\alpha_z$ be a correspondence given by Lemma \ref{p1}.  Write $\pi_{T,t_0,z}:B\to J(T)$ for the canonical projection, viewed as an algebraic correspondence. The pair given by $B$ and  $Z=\sum_{(T,t_0,z)} \alpha_z\circ \pi_{T,t_0,z}$ yields (A).

\subsection*{Explanation of (B)} It suffices to show that the map
\[f:B(k)\to A_0(B)\]
sends $l$-primary torsion to $l$-primary torsion. Let $d=\dim B$. By Bloch's theorem \cite[Th. (0.1)]{bloch}, we have $A_0(B)^{*(d+1)}=0$, where $*$ denotes Pontrjagin product. In other words, $f$ has ``degree $\le d$'' in the sense that its $(d+1)$-st deviation \cite[\S 8]{eil-ml} is identically $0$. It remains to show:

\begin{lemma}\label{l1} Let $f:M\to N$ be a map of degree $\le d$ between two abelian groups, such that $f(0)=0$. Let $m_0\in M$ be an element such that $am_0=0$ for some integer $a>0$. Then
\[a^{\binom{d+1}{2}} f(m_0)=0.\]
\end{lemma}

\begin{proof} Induction on $d$. The case $d=1$ is trivial. Assume $d>1$. By hypothesis, the $d$-th deviation of $f$ is multilinear, which implies that the map
\[g_a(m) =f(am) - a^df(m)\]
is of degree $\le d-1$. By induction, $a^{\binom{d}{2}} g_a(m_0)=0$, hence the conclusion.
\end{proof}

\begin{rk} Of course, either argument proves more generally the following: the map $\phi:A^n(V)\{l\} \to A(k)\{l\}$ is surjective  for any integer $n$, any surjective regular homomorphism $\phi:A^n(V)\to A(k)$ and any prime $l\ne \car k$. 
\end{rk}

\begin{rk} In \cite[\S 6, Lemma and Prop. 11]{beauville}, Beauville gives a different proof that $f$ sends torsion to torsion. Moreover, he observes that Ro\v\i tman's theorem \cite{roitman}  then implies that the restriction of $f$ to torsion \emph{is actually an isomorphism, hence a homomorphism}. 

If we apply Ro\v\i tman's theorem together with Lemma \ref{l1}, we obtain the following stronger result: \emph{if $m,m_0\in B(k)$ and $m_0$ is torsion, then $f(m+m_0)=f(m)+f(m_0)$}. (Fixing $m$, the map $f_m:m'\mapsto f(m+m')- f(m) - f(m')$ is of degree $< d$, hence $a^{\binom{d}{2}}f_m(m_0) = 0$ if $am_0= 0$ by Lemma \ref{l1}, and therefore $f_m(m_0)=0$ by  Ro\v\i tman's theorem.)
\end{rk}
%\enlargethispage*{20pt}
\subsection*{Some expectation} The landmark work of Bloch and Esnault \cite{be} yields the existence of $4$-folds $V$ over fields $k$ of characteristic $0$ such that the $l$-torsion of $A^3(V)$ is infinite  (hence its $l$-primary torsion has infinite corank). One example, used by Rosenschon-Srinivas \cite{rs} and Totaro \cite{totaro} and  relying on Nori's theorem \cite{nori}  and Schoen's results \cite{schoen}, is the following: start from the generic abelian $3$-fold $A$, whose field of constants $k_0$ is finitely generated over $\Q$; choose an elliptic curve $E/k_0(t)$, not isotrivial with respect to $k_0$, and take $V=A_{k_0(t)}\times E$, $k=$ algebraic closure of $k_0(t)$. %I would venture:

\begin{conj} For this $V$, a universal regular homomorphism on $A^3(V)$ does not exist.
\end{conj}

\subsection*{Ackowledgements}  I am indebted to Jacob Murre for discussions around this problem, and for his encouragement to publish this note.
 I am also indebted to the referee for a careful reading, pointing out an incorrect earlier formulation of Lemma \ref{p1}, as well as the referece to \cite{beauville}. (The referee credits in turn Charles Vial for this reference.)

\end{document}